\documentclass[11pt]{article}

\usepackage{amssymb,amsmath,euscript,dsfont,bbm}
\newtheorem{proposition}{Proposition}[section]
\newtheorem{lemma}[proposition]{Lemma}
\newtheorem{theorem}[proposition]{Theorem}

\newtheorem{corollary}[proposition]{Corollary}

\def\la{\lambda}
\def\La{\Lambda}
\def\ep{\varepsilon}

\def\l{{\langle}}
\def\r{\rangle}

\newcommand{\wt}{\widetilde}

\def\R{{\mathbb R}}

\def\E{{\mathbb E}}
\def\P{{\mathbb P}}

\makeatletter \@addtoreset{equation}{section} \makeatother
\newcommand {\qed}%
{%
    {}\hfill
    {}\hfill
    {$\square $}%
    \vspace {0.3cm}%
    \pagebreak [2]%
    \par
}%
\newenvironment{proof}[1]{%
    \vspace{0.3cm}%
    \pagebreak [2]%
    \par%
    \noindent {\bf  Proof~#1\ }}{\qed}%
\newenvironment{remark}{%
    \vspace{0.3cm} \pagebreak [2]%
    \par%
    \refstepcounter{proposition}
    \noindent%
    {\bf Remark~\theproposition\  }}{\qed}%
\begin{document}

\title {Excursion Probability of Certain Non-centered Smooth Gaussian Random Fields}
\author{Dan Cheng\\ North Carolina State University }

\maketitle

\begin{abstract}
Let $X = \{X(t): t\in T \}$ be a non-centered, unit-variance, smooth Gaussian random field indexed on some parameter space $T$, and let $A_u(X,T) = \{t\in T: X(t)\geq u\}$ be the excursion set of $X$ exceeding level $u$. Under certain smoothness and regularity conditions, it is shown that, as $u\to \infty$, the excursion probability $\P\{\sup_{t\in T} X(t)\ge u \}$ can be approximated by the expected Euler characteristic of $A_u(X,T)$, denoted by $\E\{\chi(A_u(X,T))\}$, such that the error is super-exponentially small. This verifies the expected Euler characteristic heuristic for a large class of non-centered smooth Gaussian random fields and provides a much more accurate approximation compared with those existing results by the double sum method. The explicit formulae for $\E\{\chi(A_u(X,T))\}$ are also derived for two cases: (i) $T$ is a rectangle and $X-\E X$ is stationary; (ii) $T$ is an $N$-dimensional sphere and $X-\E X$ is isotropic.
\end{abstract}



\section{Introduction}
Let $X = \{X(t): t\in T \}$ be a real-valued Gaussian random field living on some parameter space $T$. The excursion probability $\P\{\sup_{t\in T} X(t)\ge u \}$ has been extensively studied in the literature due to its importance in both theory and applications in many areas. We refer to the survey Adler (2000) and monographs Piterbarg (1996a), Adler and Taylor (2007) and Aza\"is and Wschebor (2009) for the history, recent developments and related applications on this subject. To approximate the excursion probability for high exceeding level $u$, many authors have developed various powerful tools, including the double sum method [Piterbarg (1996a)], the tube method [Sun (1993)], the expected Euler characteristic approximation [Adler (2000), Taylor and Adler (2003), Taylor et al. (2005), Adler and Taylor (2007)] and the Rice method [Aza\"is and Delmas (2002), Aza\"is and Wschebor (2008, 2009)].

In particular, the expected Euler characteristic approximation establishes a very general and profound result, building an interesting connection between the excursion probability and the geometry of the field. It was first rigorously proved by Taylor et al. (2005) [see also Theorem 14.3.3 in Adler and Taylor (2007)], showing that for a centered, unit-variance, smooth Gaussian random field, under certain conditions on the regularity of $X$ and topology of $T$,
\begin{equation}\label{Eq:MEC approx error}
\P \left\{\sup_{t\in T} X(t) \geq u \right\} = \E\{\chi(A_u(X,T))\}(1 + o\big(e^{- \alpha u^2})) \quad  {\rm as} \ u\to \infty,
\end{equation}
where $\chi(A_u(X,T))$ is the Euler characteristic of the excursion set $A_u(X,T) = \{t\in T: X(t)\geq u\}$ and $\alpha > 0$ is some constant. This verifies the ``Expected Euler Characteristic Heuristic'' for centered, unit-variance, smooth Gaussian random fields. Similar results can also be found in Aza\"is and Wschebor (2009) where the Rice method was applied. It had also been further developed by Cheng and Xiao (2014b) that \eqref{Eq:MEC approx error} holds for certain Gaussian fields with stationary increments which have nonconstant variances. However, to the best of our knowledge, among the existing works on deriving the expected Euler characteristic approximation \eqref{Eq:MEC approx error}, the Gaussian field $X$ is always assumed to be centered. In fact, the study of excursion probability for non-centered Gaussian fields is also very valuable since the varying mean function plays an important role in many models. Especially, when the Gaussian field is non-smooth, several results on the excursion probability have been obtained via the double sum method [see, for examples, Piterbarg (1996a), Piterbarg and Stamatovich (1998), Husler and Piterbarg (1999)].


In this paper, we study the excursion probability $\P\{\sup_{t\in T} X(t)\ge u \}$ for non-centered, unit variance, smooth [see condition ({\bf H}1) below] Gaussian random fields. As the first contribution, we obtain in Theorem \ref{Thm:MEC} that, in general, the expected Euler characteristic approximation \eqref{Thm:EECA} holds for such non-centered Gaussian fields when $T\subset \R^N$ is a compact rectangle. It shows that, comparing with the double sum method for non-smooth non-centered Gaussian fields [see Piterbarg and Stamatovich (1998) for example], we are able to obtain a much more accurate approximation for the excursion probability of smooth non-centered Gaussian fields such that the error is super-exponentially small. This is because the expected Euler characteristic approximation takes into account the effect of $X$ over the boundary of $T$, which is ignored in the double sum method. By similar arguments in Aza\"is and Delmas (2002), such approximation can also be easily extended to the cases when $T\subset \R^N$ is a compact and convex set with smooth boundary or a compact and smooth manifold without boundary, see Theorem \ref{Thm:EECA-T}.

To apply the approximation in practice, one needs to find an explicit formula for the expected Euler characteristic $\E\{\chi(A_u(X,T))\}$. Under the assumption of centered Gaussian fields, Taylor and Adler (2003) showed a very nice formula for $\E\{\chi(A_u(X,T))\}$ [see also Adler and Taylor (2007)] , involving the Lipschitz-Killing curvatures of the excursion set $A_u(X,T)$. However, there is lack of research to evaluate $\E\{\chi(A_u(X,T))\}$ for non-centered Gaussian fields. We provide here explicit formulae of $\E\{\chi(A_u(X,T))\}$ for two cases of non-centered Gaussian fields: (i) $T$ is a rectangle and $X-\E X$ is stationary; (ii) $T$ is an $N$-dimensional sphere and $X-\E X$ is isotropic; see respetively Theorems \ref{Thm:MEC} and \ref{Thm:MECSphere}. The results show that,  the mean function of the field does make the formula of $\E\{\chi(A_u(X,T))\}$ much more complicated than that of the centered field. In real applications, one usually needs to use the Laplace method to obtain explicit asymptotics for $\E\{\chi(A_u(X,T))\}$.

\section{Excursion Probability}

\subsection{Gaussian Random Fields on Rectangles}
We first consider the Gaussian field $X = \{X(t): t\in T \}$, where $T\subset \R^N$ is a compact rectangle. Throughout this paper, unless specified otherwise, $X$ is assumed to be unit-variance and $T$ denotes a compact rectangle. For a function $f(\cdot) \in C^2(T)$, we write $\frac{\partial f(t)}{\partial t_i} = f_i (t)$ and $\frac{\partial^2 f(t)}{\partial t_i\partial t_j} = f_{ij}(t)$. Denote by $\nabla f(t)$ and $\nabla^2 f(t)$ the column vector $(f_1(t), \ldots , f_N(t))^{T}$ and the $N\times N$ matrix $(f_{ij}(t))_{ i, j = 1, \ldots, N}$, respectively. We shall make use of the following smoothness condition ({\bf H}1) and regularity condition ({\bf H}2) for approximating the excursion probability, and also a weaker regularity condition ({\bf H}2$'$) for evaluating the expected Euler characteristic $\E\{\chi(A_u(X,T))\}$ [note that ({\bf H}2) implies ({\bf H}2$'$)].

\begin{itemize}
\item[({\bf H}1).] $X(\cdot) \in C^2(T)$ almost surely and its second derivatives satisfy the
\emph{uniform mean-square H\"older condition}: there exist constants $L>0$ and $ \eta \in (0, 1]$
such that
\begin{equation}
\E(X_{ij}(t)-X_{ij}(s))^2 \leq L d(t,s)^{2\eta}, \quad \forall t,s\in T,\ i, j= 1, \ldots, N,
\end{equation}
where $d(t,s)$ is the distance of $t$ and $s$.
\end{itemize}

\begin{itemize}
\item[({\bf H}2).]  For every pair $(t, s)\in T^2$ with $t\neq s$, the Gaussian random vector
$$(X(t), \nabla X(t), X_{ij}(t),\,
 X(s), \nabla X(s), X_{ij}(s), 1\leq i\leq j\leq N)$$
is  non-degenerate.
\end{itemize}

\begin{itemize}
\item[({\bf H}2$'$).] For every $t\in T$,  $(X(t), \nabla X(t), X_{ij}(t), 1\leq i\leq j\leq N)$
is non-degenerate.
\end{itemize}

We may write $T=\prod^N_{i=1}[a_i, b_i]$, $-\infty< a_i<b_i<\infty$. Following the notation on page 134 in Adler and Taylor (2007),  we shall show that $T$ can be decomposed into several faces of lower dimensions, based on which the Euler characteristic of the excursion set can be formulated.

A face $J$ of dimension $k$ is defined by fixing a subset $\sigma(J) \subset \{1, \ldots, N\}$ of size $k$ (If $k=0$, we have  $\sigma(J) = \emptyset$ by convention) and a subset $\varepsilon(J) = \{\varepsilon_j, j\notin \sigma(J)\} \subset \{0, 1\}^{N-k}$ of size $N-k$, so that
\begin{equation*}
\begin{split}
J= \{ t=(t_1, \ldots, t_N) \in T: \, &a_j< t_j <b_j  \ {\rm if} \ j\in \sigma(J), \\
&t_j = (1-\ep_j)a_j + \ep_{j}b_{j}  \ {\rm if} \ j\notin \sigma(J)  \}.
\end{split}
\end{equation*}
Denote by $\partial_k T$ the collection of all $k$-dimensional faces in $T$. Then
the interior of $T$ is given by $\overset{\circ}{T}=\partial_N T$ and the boundary
of $T$ is given by $\partial T = \cup^{N-1}_{k=0}\cup _{J\in \partial_k T} J$. For
$J\in \partial_k T$, denote by
$\nabla X_{|J}(t)$ and $\nabla^2 X_{|J}(t)$ the column vector $(X_{i_1} (t),
\ldots, X_{i_k} (t))^T_{i_1, \ldots, i_k \in \sigma(J)}$ and the $k\times k$ matrix
$(X_{mn}(t))_{m, n \in \sigma(J)}$, respectively.

If $X(\cdot)\in C^2(T)$ and it is a Morse function a.s. [cf. Definition 9.3.1 in Adler and Taylor (2007)], then according to Corollary 9.3.5 or pages 211-212 in Adler and Taylor (2007), the Euler characteristic of the excursion set $A_u(X,T) = \{t\in T: X(t)\geq u\}$ is given by
\begin{equation}\label{Eq:def of Euler charac}
\chi(A_u(X,T))= \sum^N_{k=0}\sum_{J\in \partial_k T}(-1)^k\sum^k_{i=0} (-1)^i \mu_i(J)
\end{equation}
with
\begin{equation*}
\begin{split}
\mu_i(J) & := \# \big\{ t\in J: X(t)\geq u, \nabla X_{|J}(t)=0, \text{index} (\nabla^2 X_{|J}(t))=i, \\
  & \qquad \quad \varepsilon^*_jX_j(t) \geq 0 \ {\rm for \ all}\ j\notin \sigma(J) \big\},
\end{split}
\end{equation*}
where $\ep^*_j=2\ep_j-1$ and the index of a matrix is defined as the number of its negative eigenvalues.

For $t\in J\in \partial_k T$, let
\begin{equation}\label{Def:E(J)}
\begin{split}
\La_J(t)=(\la_{ij}(t))_{i,j\in \sigma(J)}&:= ({\rm Cov}(X_i(t), X_j(t)))_{i,j\in \sigma(J)} = {\rm Cov}(\nabla X_{|J}(t), \nabla X_{|J}(t)), \\
 \{J_1, \cdots, J_{N-k}\} &= \{1, \cdots, N\}\backslash \sigma(J),\\
E(J) &= \{(t_{J_1}, \cdots, t_{J_{N-k}})\in \R^{N-k}: t_j\ep_j^*\geq 0,  j= J_1, \cdots, J_{N-k}\}.\\
\end{split}
\end{equation}
Since $X$ has unit variance, ${\rm Cov}(X(t), \nabla^2 X_{|J}(t)) = -{\rm Cov}(\nabla X_{|J}(t), \nabla X_{|J}(t)) = -\La_J(t)$, which is negative definite. Define the number of \emph{extended outward maxima}
above level $u$ as
\begin{equation*}
\begin{split}
M_u^E (J) & := \# \big\{ t\in J: X(t)\geq u, \nabla X_{|J}(t)=0, \text{index} (\nabla^2 X_{|J}(t))=k, \\
  & \qquad \qquad \qquad \varepsilon^*_jX_j(t) \geq 0 \ {\rm for \ all}\ j\notin \sigma(J) \big\}\\
& = \big\{ t\in J: X(t)\geq u, \nabla X_{|J}(t)=0, \text{index} (\nabla^2 X_{|J}(t))=k, \\
  & \qquad \qquad \quad (X_{J_1}(t), \cdots, X_{J_{N-k}}(t))\in E(J) \big\}.
\end{split}
\end{equation*}
By similar arguments in Piterbarg (1996b) or Cheng and Xiao (2014b), we have the following bounds for the excursion probability:
\begin{equation}\label{Ineq:bounds}
\begin{split}
\sum_{k=0}^N\sum_{J\in\partial_k T} \E \{M_u^E (J) \} &\ge \P \left\{\sup_{t\in T} X(t) \geq u \right\} \\
& \ge \sum_{k=0}^N\sum_{J\in\partial_k T} \left ( \E \{M_u^E (J) \}
- \frac{1}{2}\E \{M_u^E (J)(M_u^E (J) -1) \} \right ) \\
& \quad - \sum_{J\neq J'} \E\{M_u^E (J) M_u^E (J')\}.
\end{split}
\end{equation}

We call a function $h(u)$ \emph{super-exponentially small} [when compared with  $\P (\sup_{t\in T} X(t) \geq u )$], if there exists a constant $\alpha >0$ such that $h(u) = o(e^{-\alpha u^2 - u^2/2})$ as $u \to \infty$. The sketch for proving the expected Euler characteristic approximation \eqref{Eq:MEC approx error} consists of two steps. The first step, which is established in Lemma \ref{Lem:approximateMuE} below,  is to show that the difference between the upper bound in \eqref{Ineq:bounds} and the expected Euler characteristic $\E\{\chi(A_u(X,T))\}$ is super-exponentially small. Then we prove that the upper bound in \eqref{Ineq:bounds} makes the major contribution since the last two terms in the lower bound in \eqref{Ineq:bounds} are super-exponentially small, see Corollary \ref{Cor:Piterbarg} and Lemmas \ref{Lem:disjoint faces} and \ref{Lem:adjacent faces} below.

\begin{lemma}\label{Lem:approximateMuE}
Let $X = \{X(t): t\in T \}$ be a Gaussian random field satisfying $({\bf H}1)$ and $({\bf H}2')$. Then for each
$J\in \partial_k T$ with $k\geq 1$, there exists some constant $\alpha>0$ such that
\begin{equation}\label{Eq:major-term}
\E \left\{M_u^E (J) \right\} = \E \left\{(-1)^k\sum^k_{i=0} (-1)^i \mu_i(J)\right\} (1+o(e^{-\alpha u^2})).
\end{equation}
\end{lemma}
\begin{proof}
\ To simplify the notation, without loss of generality, we assume $\sigma(J)= \{1, \ldots, k\}$ and that all elements in $\ep(J)$ are 1, which implies $E(J)=\R^{N-k}_+$. Let $\mathcal{D}_i$ be the collection of all $k \times k $ matrices with index $i$. By the Kac-Rice metatheorem [cf. Theorem 11.2.1 or Corollary 11.2.2 in Adler
and Taylor (2007)], $\E \{M_u^E (J)\}$ equals
\begin{equation}\label{Eq:mean-mu}
\begin{split}
&\int_J \E\{ |\text{det} \nabla^2 X_{|J}(t)| \mathbbm{1}_{\{\nabla^2 X_{|J}(t) \in \mathcal{D}_k\}}\mathbbm{1}_{\{X(t)\geq u\}}\mathbbm{1}_{\{ (X_{k+1}(t), \cdots, X_{N}(t))\in \R^{N-k}_+\}}  | \nabla X_{|J}(t)=0 \}\\
&\quad \times p_{\nabla X_{|J}(t)}(0)  dt\\
&= (-1)^k \int_J \, dt\int^\infty_u dx\int_0^\infty dy_{k+1} \cdots \int _0^\infty dy_{N}\\
&\quad \E\{\text{det} \nabla^2 X_{|J}(t) \mathbbm{1}_{\{\nabla^2 X_{|J}(t) \in \mathcal{D}_k\}} | X(t)=x, X_{k+1}(t)=y_{k+1}, \cdots, X_{N}=y_{N}, \nabla X_{|J}(t)=0 \} \\
&\quad \times p_{X(t), X_{k+1}(t), \cdots, X_{N}(t)}(x, y_{k+1}, \cdots, y_{N}| \nabla X_{|J}(t)=0) p_{\nabla X_{|J}(t)}(0).
\end{split}
\end{equation}

Since $\La_J(t)$ is positive definite for every $t\in J$, there exists a $k \times k $ positive definite matrix $Q_t$ such that
$Q_t\La_J(t)Q_t = I_k$, where $I_k$ is the $k\times k$ identity matrix. We write $\nabla^2 X_{|J}(t)= Q_t^{-1}Q_t\nabla^2 X_{|J}(t)Q_tQ_t^{-1}$ and let $a_{ij}^l(t) = {\rm Cov}(X_l(t), (Q_t \nabla^2X_{|J}(t) Q_t)_{ij})$ for $l= 1, \cdots, N$. By the well-known conditional formula for Gaussian variables,
\begin{equation}\label{Eq:EQX}
\begin{split}
\E &\{(Q_t\nabla^2 X_{|J}(t)Q_t)_{ij} | X(t)=x, \nabla X_{|J}(t)=0, X_{k+1}(t)=y_{k+1}, \ldots, X_{N}(t)=y_{N} \}\\
&=(Q_t\nabla^2 m_{|J}(t)Q_t)_{ij} + (-\delta_{ij}, a_{ij}^1(t), \ldots, a_{ij}^N(t)) ( \text{Cov} (X(t), \nabla X(t)) )^{-1} \\
& \quad \cdot (x, 0, \ldots, 0, y_{k+1}, \cdots, y_N)^T.
\end{split}
\end{equation}
Make change of variables $V(t)= (V_{ij}(t))_{1\leq i, j \leq k}$, where
\begin{equation*}
V_{ij}(t) = (Q_t\nabla^2 X_{|J}(t)Q_t)_{ij} - (Q_t\nabla^2 m_{|J}(t)Q_t)_{ij} + x\delta_{ij},
\end{equation*}
i.e.,
\begin{equation}\label{Eq:ChangeVariables}
Q_t\nabla^2 X_{|J}(t)Q_t = V(t)  + Q_t\nabla^2 m_{|J}(t)Q_t - xI_k.
\end{equation}
Denote the density of
$$((V_{ij}(t))_{1\leq i\leq j\leq k}|X(t)=x, \nabla X_{|J}(t)=0, X_{k+1}(t)=y_{k+1}, \cdots, X_{N}(t)=y_{N})$$
by $h_{t, y_{k+1}, \ldots, y_N}(v)$, $v = (v_{ij}: 1\leq i \leq j\leq k)\in \R^{k(k+1)/2}$. It follows from \eqref{Eq:EQX} and the independence of $X(t)$ and $\nabla X(t)$ that $h_{t, y_{k+1}, \ldots, y_N}(v)$ is independent of $x$. Let $(v_{ij})$ be the abbreviation of matrix $(v_{ij})_{1\leq i, j\leq k}$. Applying \eqref{Eq:ChangeVariables} yields
\begin{equation}\label{Eq:mean critical points 2}
\begin{split}
&\quad \E\{{\rm det} (Q_t\nabla^2 X_{|J}(t)Q_t)\mathbbm{1}_{\{\nabla^2 X_{|J}(t)\in \mathcal{D}_k\}} | X(t)=x, \nabla X_{|J}(t)=0, \\
& \qquad \qquad \qquad \qquad \qquad \qquad \qquad \qquad \quad  X_{k+1}(t)=y_{k+1}, \cdots, X_{N}(t)=y_{N} \}\\
& = \E\{{\rm det} (Q_t\nabla^2 X_{|J}(t)Q_t) \mathbbm{1}_{\{Q_t\nabla^2 X_{|J}(t)Q_t\in \mathcal{D}_k\}} | X(t)=x,  \nabla X_{|J}(t)=0, \\
& \qquad \qquad \qquad \qquad \qquad \qquad \qquad \qquad \qquad  \quad \, X_{k+1}(t)=y_{k+1}, \cdots, X_{N}(t)=y_{N}\}\\
& = \int_{\{v: (v_{ij})+ Q_t\nabla^2 m_{|J}(t)Q_t - xI_k \in \mathcal{D}_k\}} {\rm det} \left((v_{ij})+ Q_t\nabla^2 m_{|J}(t)Q_t - xI_k\right) h_{t, y_{k+1}, \ldots, y_N}(v) \, dv.
\end{split}
\end{equation}
Since $Q_t\nabla^2 m_{|J}(t)Q_t$ is continuous in $t$ and $T$ is compact, there exists some constant $c>0$ such that the following relation holds for all $t\in T$ and $x$ large enough:
\begin{equation}\label{Eq:VDk}
(v_{ij})+ Q_t\nabla^2 m_{|J}(t)Q_t - xI_k \in \mathcal{D}_k, \qquad \forall \ \|(v_{ij})\| <\frac{x}{c}.
\end{equation}
Let
\begin{equation*}
\begin{split}
W&(t,x,y_{k+1}, \cdots, y_N) \\
&= \int_{\{v: (v_{ij})+ Q_t\nabla^2 m_{|J}(t)Q_t - xI_k \notin \mathcal{D}_k\}} {\rm det} \left((v_{ij})+ Q_t\nabla^2 m_{|J}(t)Q_t - xI_k\right) h_{t, y_{k+1}, \ldots, y_N}(v) \, dv.
\end{split}
\end{equation*}
Then \eqref{Eq:mean critical points 2} becomes
\begin{equation}\label{Eq:mean critical points 3}
\begin{split}
\int_{\R^{k(k+1)/2}} {\rm det} \left((v_{ij})+ Q_t\nabla^2 m_{|J}(t)Q_t - xI_k\right)h_{t, y_{k+1}, \ldots, y_N}(v) \, dv - W(t,x,y_{k+1}, \cdots, y_N).
\end{split}
\end{equation}
It follows from \eqref{Eq:VDk} that
\begin{equation*}
\begin{split}
I(t,x)&:= \int_0^\infty dy_{k+1} \cdots \int _0^\infty dy_{N}\, p_{X(t), X_{k+1}(t), \cdots, X_{N}(t)}(x, y_{k+1}, \cdots, y_{N}| \nabla X_{|J}(t)=0)\\
&\quad \times |W(t,x,y_{k+1}, \cdots, y_{N})| \\
&\leq \int_0^\infty dy_{k+1} \cdots \int _0^\infty dy_{N}\, p_{X(t), X_{k+1}(t), \cdots, X_{N}(t)}(x, y_{k+1}, \cdots, y_{N}| \nabla X_{|J}(t)=0)  \\
&\quad \times \int_{\|(v_{ij})\|\geq\frac{x}{c}} \left|\text{det} \left((v_{ij})+ Q_t\nabla^2 m_{|J}(t)Q_t-xI_k \right)\right| h_{t,y_{k+1}, \cdots, y_{N}}(v)  dv\\
&\leq p_{X(t)}(x| \nabla X_{|J}(t)=0) \int_{\|(v_{ij})\|\geq\frac{x}{c}} \left|\text{det} \left((v_{ij})+ Q_t\nabla^2 m_{|J}(t)Q_t-xI_k \right)\right| f_t(v)  dv,
\end{split}
\end{equation*}
where $f_t(v)$ is the density of $((V_{ij}(t))_{1\leq i\leq j\leq k}|X(t)=x, \nabla X_{|J}(t)=0)$ and the last inequality comes from replacing the integral domain $\R_+^{N-k}$ by $\R^{N-k}$. Hence there exists some $\alpha>0$ such that $\int_J \int^\infty_u I(t,x)dxdt = o(e^{-\alpha u^2 - u^2/2})$ as $u \to \infty$. Plugging this, together with \eqref{Eq:mean critical points 2} and \eqref{Eq:mean critical points 3}, into \eqref{Eq:mean-mu}, we see that $\E \{M_u^E (J)\}$ becomes
\begin{equation*}
\begin{split}
&(-1)^k \int_J  {\rm det}(\La_J(t))dt\int^\infty_u dx\int_0^\infty dy_{k+1} \cdots \int _0^\infty dy_{N}\\
&\E\{\text{det} (Q_t\nabla^2 X_{|J}(t)Q_t) \mathbbm{1}_{\{\nabla^2 X_{|J}(t) \in \mathcal{D}_k\}} | X(t)=x, X_{k+1}(t)=y_{k+1}, \cdots, X_{N}=y_{N}, \nabla X_{|J}(t)=0 \} \\
&\times p_{X(t), X_{k+1}(t), \cdots, X_{N}(t)}(x, y_{k+1}, \cdots, y_{N}| \nabla X_{|J}(t)=0) p_{\nabla X_{|J}(t)}(0)  dt.\\
&\quad = (-1)^k \bigg[\int_J dt\int^\infty_u dx\int_0^\infty dy_{k+1} \cdots \int _0^\infty dy_{N}\\
&\qquad \E\{\text{det} \nabla^2 X_{|J}(t) | X(t)=x, X_{k+1}(t)=y_{k+1}, \cdots, X_{N}=y_{N}, \nabla X_{|J}(t)=0 \} \\
&\qquad \times p_{X(t), X_{k+1}(t), \cdots, X_{N}(t)}(x, y_{k+1}, \cdots, y_{N}| \nabla X_{|J}(t)=0) p_{\nabla X_{|J}(t)}(0) \bigg] + o(e^{-\alpha u^2 - u^2/2})\\
&\quad = \E \left\{(-1)^k\sum^k_{i=0} (-1)^i \mu_i(J)\right\} (1+o(e^{-\alpha u^2})),
\end{split}
\end{equation*}
where the last line is due to the Kac-Rice metatheorem and the fact that
$$
\sum^k_{i=0}(-1)^i|\text{det} \nabla^2 X_{|J}(t)|\mathbbm{1}_{\{\nabla^2 X_{|J}(t) \in \mathcal{D}_i\}}= {\rm det} \nabla^2 X_{|J}(t), \quad {\rm a.s.}
$$
\end{proof}

Lemma \ref{Lem:Piterbarg} below can be derived from Lemma 4 in Piterbarg (1996b). It will be used to show in Corollary \ref{Cor:Piterbarg} that the factorial moments of $M_u ^E(J)$ are usually super-exponentially small.
\begin{lemma} \label{Lem:Piterbarg}
Let $X = \{X(t): t\in T \}$ be a Gaussian random field satisfying $({\bf H}1)$ and $({\bf H}2)$. Then for any $\ep >0$, there exists $\varepsilon_1 >0$ such that for any $J\in \partial_k T$ with $k\ge 1$ and $u$ large enough,
\begin{equation*}
\begin{split}
\E \{M_u^E (J)(M_u^E(J)- 1)\} \leq e^{-u^2/(2\beta_J^2 +\varepsilon)} + e^{-u^2/(2 -\ep_1)},
\end{split}
\end{equation*}
where $\beta_J^2 = \sup_{t\in J} \sup_{e\in \mathbb{S}^{k-1}} {\rm Var} (X(t)|\nabla X_{|J}(t), \nabla^2 X_{|J}(t)e)$. Here and in the sequel, $\mathbb{S}^{k-1}$
 is the unit sphere in $\R^{k}$ .
\end{lemma}

\begin{corollary}\label{Cor:Piterbarg}
Let $X = \{X(t): t\in T \}$ be a Gaussian random field satisfying $({\bf H}1)$ and $({\bf H}2)$. Then for all $J\in \partial_k T$, $\E \{M_u^E (J)(M_u^E (J) -1) \}$ are super-exponentially small.
\end{corollary}

\begin{proof}\ If $k=0$, then $M_u^E (J)$ is either 0 or 1 and hence $\E \{M_u^E (J)(M_u^E (J) -1) \}=0$. If $k\geq 1$, then, thanks to Lemma \ref{Lem:Piterbarg}, it suffices to show $\beta_J^2<1$. Clearly, for every $e \in \mathbb{S}^{k-1}$ and $t \in T$, $\text{Var} (X(t)|\nabla X_{|J}(t), \nabla^2 X_{|J}(t)e) \leq 1$. On the other hand,
\begin{equation} \label{Eq:contra}
\begin{split}
&\quad \ \text{Var} (X(t)|\nabla X_{|J}(t), \nabla^2 X_{|J}(t)e) = 1 \ \Longrightarrow\ {\rm Cov} (X(t), \nabla^2 X_{|J}(t)e)=0.
\end{split}
\end{equation}
Note that  the right hand side of (\ref{Eq:contra}) is equivalent to $\La_J(t)e=0$. However, by ({\bf H}2), $\La_J(t)$ is positive definite, which implies $\La_J(t)e\neq 0$ for all $e \in \mathbb{S}^{k-1}$.  Thus for every $e \in \mathbb{S}^{k-1}$ and $t \in T$, $\text{Var} (X(t)|\nabla X_{|J}(t), \nabla^2 X_{|J}(t)e) < 1$. Combining this with the continuity of $\text{Var} (X(t)|\nabla X_{|J}(t), \nabla^2 X_{|J}(t)e)$
in $(e, t)$, we conclude $\beta_J^2 < 1$.
\end{proof}

By similar arguments for showing Lemma 4.5 in Cheng and Xiao (2014b), one can easily obtain that the cross terms $\E \{M_u^E (J) M_u^E (J') \}$ in (\ref{Ineq:bounds}) are super-exponentially small if $J$ and $J'$ are not adjacent. In particular, as the main step therein, Eq. (4.13) is essentially not affected by the mean function of the field. We thus have the following result.
\begin{lemma}\label{Lem:disjoint faces}
Let $X = \{X(t): t\in T \}$ be a Gaussian random field satisfying $({\bf H}1)$ and $({\bf H}2)$. Let $J$ and $J'$ be two faces of $T$ such that their distance is positive, i.e., $\inf_{t\in J, s\in J'}\|s-t\|>\delta_0$ for some $\delta_0>0$. Then $\E \{M_u^E (J) M_u^E (J') \}$ is super-exponentially small.
\end{lemma}

Next we turn to the alternative case when $J$ and $J'$ are adjacent. In such case, it is more technical to prove that $\E \{M_u^E (J) M_u^E (J') \}$ is super-exponentially small. To shorten the arguments for deriving Lemma \ref{Lem:adjacent faces} below, we will quote certain similar results in the proof of Theorem 4.8 in Cheng and Xiao (2014b) [or Theorem 4 in Aza\"is and Delmas (2002)].
\begin{lemma}\label{Lem:adjacent faces}
Let $X = \{X(t): t\in T \}$ be a Gaussian random field satisfying $({\bf H}1)$ and $({\bf H}2)$. Let $J$ and $J'$ be two faces of $T$ such that they are adjacent, i.e., $\inf_{t\in J, s\in J'}\|s-t\|=0$. Then $\E \{M_u^E (J) M_u^E (J') \}$ is super-exponentially small.
\end{lemma}

\begin{proof}\,  Let $I:= \bar{J}\cap \bar{J'} \neq \emptyset$. Without loss of generality, we assume
\begin{equation}\label{Eq:assumption for faces}
\begin{split}
\sigma(J)= \{1, \ldots, l, l+1, \ldots, k\},\ \
\sigma(J')= \{1, \ldots, l, k+1, \ldots, k+k'-l\},
\end{split}
\end{equation}
where $0 \leq l \leq k \leq k' \leq N$ and $k'\geq 1$. Recall that, if $k=0$, then $\sigma(J)
= \emptyset$. Under assumption (\ref{Eq:assumption for faces}), we have $J\in \partial_k T$, $J'\in \partial_{k'} T$
and $\text{dim} (I) =l$. Assume also that all elements in $\ep(J)$ and $\ep(J')$ are 1, which implies $E(J)=\R^{N-k}_+$ and $E(J')=\R^{N-k'}_+$.

We first consider the case $k\geq 1$. By the Kac-Rice metatheorem, $\E \{M_u^E (J) M_u^E (J') \}$ is bounded from above by
\begin{equation}\label{Eq:cross term}
\begin{split}
&\int_{J} dt\int_{J'} ds \int_u^\infty dx \int_u^\infty dy  \int_0^\infty dz_{k+1} \cdots \int_0^\infty dz_{k+k'-l} \int_0^\infty dw_{l+1} \cdots \int_0^\infty dw_{k}\\
& \quad \E \big\{ |\text{det} \nabla^2 X_{|J}(t) ||\text{det} \nabla^2 X_{|J'}(s) | \big| X(t)=x, X(s)=y,\nabla X_{|J}(t)=0,  X_{k+1}(t)=z_{k+1},  \\
& \qquad  \ldots, X_{k+k'-l}(t)=z_{k+k'-l}, \nabla X_{|J'}(s)=0, X_{l+1}(s)=w_{l+1}, \ldots, X_{k}(s)=w_{k} \big\}\\
& \quad \times p_{t,s}(x,y,0, z_{k+1}, \ldots,z_{k+k'-l}, 0, w_{l+1}, \ldots, w_{k} )\\
&:= \int \int_{J\times J'} A(t,s)\,dtds,
\end{split}
\end{equation}
where $p_{t,s}(x,y,0, z_{k+1}, \ldots,z_{k+k'-l}, 0,w_{l+1}, \ldots, w_{k} )$ is the density of
$$(X(t),X(s),\nabla X_{|J}(t), X_{k+1}(t),\ldots, X_{k+k'-l}(t), \nabla X_{|J'}(s), X_{l+1}(s), \ldots, X_{k}(s))$$
evaluated at $(x,y,0, z_{k+1}, \ldots,z_{k+k'-l}, 0,w_{l+1}, \ldots, w_{k} )$.

Let $\{e_1, e_2, \ldots, e_N\}$ be the standard orthonormal basis of $\R^N$. For $t\in J$ and $s\in J'$, let $e_{t, s}=(s-t)^T/\|s-t\|$ and let $\alpha_i(t, s)= \l e_i, \La(t)e_{t,s}\r$. Then
\begin{equation*}
\La(t)e_{t,s}=\sum_{i=1}^N \l e_i, \La(t)e_{t,s}\r e_i = \sum_{i=1}^N  \alpha_i(t, s) e_i.
\end{equation*}
Since $\La(t)$ are uniformly positive definite for all $t\in T$, there exists some $\alpha_0 >0$ such that
\begin{equation*}
\l e_{t,s}, \La(t)e_{t,s} \r \geq \alpha_0
\end{equation*}
for all $t$ and $s$. Let
\begin{equation*}
\begin{split}
D_i &= \{ (t,s)\in J\times J': \alpha_i (t,s)\geq \beta_i \}, \quad \text{if}\ l+1 \leq i \leq k,\\
D_i &= \{ (t,s)\in J\times J': \alpha_i (t,s)\leq -\beta_i \}, \quad \text{if}\ k+1\leq i \leq k+k'-l,\\
D_0 &= \bigg\{ (t,s)\in J\times J': \sum_{i=1}^l \alpha_i (t,s)\l e_i , e_{t,s}\r\geq \beta_0\bigg \},
\end{split}
\end{equation*}
where $\beta_0, \beta_1,\ldots, \beta_{k+k'-l}$ are positive constants such that $\beta_0 + \sum_{i=l+1}^{k+k'-l} \beta_i < \alpha_0$. Similarly to the proof of Theorem 4.8 in Cheng and Xiao (2014b), $D_0\cup \cup_{i=l+1}^{k+k'-l} D_i$ is a covering of $J\times J'$. By (\ref{Eq:cross term}),
\begin{equation*}
\E \{M_u^E (J) M_u^E (J') \} \leq \int \int_{D_0} A(t,s)\,dtds + \sum_{i=l+1}^{k+k'-l} \int \int_{D_i} A(t,s)\,dtds.
\end{equation*}

It can be shown similarly to the proof of Theorem 4.8 in Cheng and Xiao (2014b) that $\int \int_{D_0} A(t,s)\,dtds$ is super-exponentially small. Next we show that $\int \int_{D_i} A(t,s)\,dtds$ is super-exponentially small for $i= l+1, \ldots, k$. 

It follows from (\ref{Eq:cross term}) that $\int \int_{D_i} A(t,s)\,dtds$ is bounded above by
\begin{equation}\label{Eq:integral on Di}
\begin{split}
&\int \int_{D_i} dtds \int_u^\infty dx \int_0^\infty dw_i \, p_{X(t), \nabla X_{|J}(t), X_i(s), \nabla X_{|J'} (s)}(x,0,w_i,0)\\
&\times \E \{ |\text{det} \nabla^2 X_{|J}(t)| |\text{det} \nabla^2 X_{|J'}(s)|  | X(t)=x, \nabla X_{|J}(t)=0 , X_i(s)=w_i, \nabla X_{|J'} (s)=0 \}.
\end{split}
\end{equation}
Notice that if a subset $B \subset D_i$ satisfies $\inf_{t\in B\cap J,\, s\in B\cap J'}\|s-t\|>\eta_0$ for some $\eta_0>0$, then similarly to Lemma \ref{Lem:disjoint faces}, $\int \int_B A(t,s)\,dtds$ is super-exponentially small. Therefore, in the arguments below, we only treat the case when $t$ and $s$ are close enough or $\|t-s\| \to 0$.

There exists some positive constant $C_1$ such that
\begin{equation}\label{Eq:integral on Di 2}
\begin{split}
&p_{X(t), \nabla X_{|J}(t), X_i(s), \nabla X_{|J'} (s)}(x,0,w_i,0)\\
&\quad = p_{\nabla X_{|J'} (s), X_1(t),\ldots,X_{i-1}(t), X_{i+1}(t),\ldots, X_k(t)}(0|X(t)=x, X_i(s)=w_i,X_i(t)=0)\\
&\quad \quad \times p_{X(t)}(x|X_i(s)=w_i, X_i(t)=0)p_{X_i(s)}(w_i|X_i(t)=0)p_{X_i(t)}(0)\\
&\quad \leq C_1(\text {detCov} (X(t), \nabla X_{|J}(t), X_i(s), \nabla X_{|J'} (s)))^{-1/2}\\
&\qquad \times \exp \bigg\{-\frac{(x-\xi_2(t,s))^2}{2\sigma_2^2(t,s)}\bigg\}\exp \bigg\{-\frac{(w_i-\xi_1(t,s))^2}{2\sigma_1^2(t,s)}\bigg\},
\end{split}
\end{equation}
where
\begin{equation*}
\begin{split}
\xi_1(t,s) &= \E\{X_i(s) | X_i(t)=0\} = m_i(s),\\
\sigma_1^2(t,s) &= {\rm Var}(X_i(s) | X_i(t)=0) = \frac{{\rm det Cov} (X_i(s), X_i(t))}{\la_{ii}(t)},\\
\xi_2(t,s) &= \E\{ X(t) | X_i(s)=w_i, X_i(t)=0\},\\
\sigma_2^2(t,s) &= {\rm Var}( X(t) | X_i(s)=w_i, X_i(t)=0).
\end{split}
\end{equation*}
In particular, applying Taylor's formula to $X_i(s)$ [see Eq. (4.23) in Cheng and Xiao (2014b) or Piterbarg (1996b)], one has
\begin{equation}\label{Eq:xi2}
\begin{split}
\xi_2(t,s) &= \E\{ X(t) | \l\nabla X_i(t), e_{t,s}\r =w_i/\|s-t\| +o(1), X_i(t)=0\},\\
&= m(t) + ({\rm Cov}(X(t), \l\nabla X_i(t), e_{t,s}\r), 0)\left(
            \begin{array}{cc}
              \frac{1}{{\rm Var}(\l\nabla X_i(t), e_{t,s}\r)} & 0 \\
              0 & \frac{1}{\la_{ii}(t)}\\
            \end{array}
          \right)\\
          &\quad \cdot \left(
                   \begin{array}{c}
                     w_i/\|s-t\| +o(1) -  \l\nabla m_i(t), e_{t,s}\r\\
                     -m_i(t) \\
                   \end{array}
                 \right)\\
&= m(t) - \frac{\alpha_i(t,s)[w_i/\|s-t\| -  \l\nabla m_i(t), e_{t,s}\r+o(1)]}{{\rm Var}(\l\nabla X_i(t), e_{t,s}\r)}
\end{split}
\end{equation}
and
\begin{equation}\label{Eq:sigma2}
\begin{split}
\sigma_2^2(t,s) = {\rm Var}( X(t) | \l\nabla X_i(t), e_{t,s}\r, X_i(t)) + o(1) \le 1-\delta_0
\end{split}
\end{equation}
for some $\delta_0>0$.

Also, by similar arguments in the proof of Theorem 4.8 in Cheng and Xiao (2014b), there exist positive constants $C_2, C_3, N_1$ and $N_2$ such that
\begin{equation}\label{Eq:integral on Di 3}
{\rm det Cov} (\nabla X_{|J} (t), X_i(s), \nabla X_{|J'} (s)) \geq C_2 \|s-t\|^{2(l+1)}
\end{equation}
and
\begin{equation}\label{Eq:integral on Di 4}
\begin{split}
\E &\left\{ |{\rm det} \nabla^2 X_{|J} (t) | |{\rm det} \nabla^2 X_{|J'}(s)| \big| X(t)=x, \nabla X_{|J}(t)=0 , X_i(s)=w_i, \nabla X_{|J'} (s)=0 \right\}\\
&=\E \big\{ |{\rm det} \nabla^2 X_{|J} (t) | |{\rm det} \nabla^2 X_{|J'}(s)| \big| X(t)=x, \nabla X_{|J}(t)=0 , \\
&\qquad \ \l\nabla X_i(t), e_{t,s}\r =w_i/\|s-t\| +o(1), \nabla X_{|J'} (s)=0 \big\}\\
&\leq C_3(x^{N_1}+|w_i/\|s-t\||^{N_2} +1 ).
\end{split}
\end{equation}
Combining (\ref{Eq:integral on Di}), \eqref{Eq:integral on Di 2}, \eqref{Eq:xi2}, \eqref{Eq:sigma2}, (\ref{Eq:integral on Di 3}) and (\ref{Eq:integral on Di 4}), and making change of variable $w= w_i/\|s-t\|$, we obtain that there exist positive constants $C_4, C_5$, $C_6$ and $C_7$ such that $\int \int_{D_i} A(t,s)\,dtds$ is bounded above by
\begin{equation*}
\begin{split}
&\quad C_4 \int \int_{D_i} dtds \|s-t\|^{-l-1} \int_u^\infty dx \int_0^\infty dw_i(x^{N_1}+|w_i/\|s-t\||^{N_2} +1 )\\
&\qquad \times \exp \left\{-\frac{(x-\xi_2(t,s))^2}{2\sigma_2^2(t,s)}\right\}\exp \bigg\{-\frac{(w_i-m_i(s))^2}{2\sigma_1^2(t,s)}\bigg\} \\
&= C_4\int \int_{D_i} dtds \|s-t\|^{-l} \int_u^\infty dx \int_0^\infty dw (x^{N_1}+|w|^{N_2} +1 )\\
&\quad \times \exp \left\{-\frac{\left(x-m(t) + \frac{\alpha_i(t,s)[w -  \l\nabla m_i(t), e_{t,s}\r+o(1)]}{{\rm Var}(\l\nabla X_i(t), e_{t,s}\r)}   \right)^2}{2\sigma_2^2(t,s)}\right\}\exp \bigg\{-\frac{(w-\wt{m}_i(t,s))^2}{2\wt{\sigma}_1^2(t,s)}\bigg\}\\
&\le C_4\int \int_{D_i} dtds \|s-t\|^{-l} \int_u^\infty \exp \bigg\{-\frac{\left[x-C_5 +  \beta_i (C_6w-C_7) \right]^2}{2(1-\delta_0)}\bigg\} dx \\
&\quad \times \int_0^\infty (x^{N_1}+|w|^{N_2} +1 )\exp \bigg\{-\frac{(w-\wt{m}_i(t,s))^2}{2\wt{\sigma}_1^2(t,s)}\bigg\}dw,
\end{split}
\end{equation*}
where $\wt{\sigma}_1(t,s) = \sigma_1(t,s)/\|s-t\|$, $\wt{m}_i(t,s) = m_i(s)/\|s-t\|$ and we have used the fact $\alpha_i (t,s)\geq \beta_i >0$ for the last line. This, in turn, ensures that there exists some $\delta_1\in(0,\delta_0)$ such that for sufficiently large $u$,
\begin{equation*}
\begin{split}
&\quad \int \int_{D_i} A(t,s)\,dtds\\
&\le C_4\exp \bigg\{-\frac{u^2}{2(1-\delta_1)}\bigg\}\int \int_{D_i}  \|s-t\|^{-l}dtds \\
&\quad \times\int_0^\infty \exp \bigg\{-\frac{\beta_i^2 C_6^2w^2}{2(1-\delta_1)}\bigg\} (u^{N_1}+|w|^{N_2} +1 )\exp \bigg\{-\frac{(w-\wt{m}_i(t,s))^2}{2\wt{\sigma}_1^2(t,s)}\bigg\}dw\\
&\le C_4\exp \bigg\{-\frac{u^2}{2(1-\delta_1)}\bigg\}\int \int_{D_i}  \|s-t\|^{-l}dtds\int_0^\infty \exp \bigg\{-\frac{\beta_i^2 C_6^2w^2}{2(1-\delta_1)}\bigg\} (u^{N_1}+|w|^{N_2} +1 )dw.
\end{split}
\end{equation*}
Since $\|s-t\|^{-l}$ is integrable on $J\times J'$, we conclude that $\int \int_{D_i} A(t,s)\,dtds$ is finite and super-exponentially small.

It is similar to show that $\int \int_{D_i} A(t,s)\,dtds$ is super-exponentially small for $i=k+1, \ldots, k+k'-l$. The case when $k=0$ can also be proved similarly.
\end{proof}

Now we can derive our main result of this section.
\begin{theorem} \label{Thm:EECA}
Let $X = \{X(t): t\in T \}$ be a Gaussian random field satisfying $({\bf H}1)$ and $({\bf H}2)$. Then there exists some $\alpha>0$ such that the expected Euler characteristic approximation \eqref{Eq:MEC approx error} holds.
\end{theorem}
\begin{proof}\
The result follows immediately from combining \eqref{Ineq:bounds}, Lemma \ref{Lem:approximateMuE}, Corollary \ref{Cor:Piterbarg}, Lemma \ref{Lem:disjoint faces} and Lemma \ref{Lem:adjacent faces}.
\end{proof}

\subsection{Gaussian Random Fields on Other Sets}
Adler and Taylor (2007) obtained the expected Euler characteristic approximation \eqref{Eq:MEC approx error} for centered Gaussian fields living on quite general manifolds. Since the method used in this paper is different, and it may require more powerful techniques and careful arguments to extend the parameter sets to general manifolds, hence we will not attempt to achieve such extension here. However, similarly to Aza\"is and Delmas (2002), we can easily extend the approximation to the cases of smooth and compact manifolds without boundary or convex and compact sets with smooth boundary.

We first introduce some notation. Let $(T,g)$ be an $N$-dimensional Riemannian manifold, where $g$ is the Riemannian metric, and let $f$ be a real-valued smooth function on $T$. Then the \emph{gradient} of $f$, denoted by $\nabla f$, is the unique continuous vector field on $T$ such that $g(\nabla f, \xi) = \xi f$ for every vector field $\xi$. The \emph{Hessian} of $f$, denoted by $\nabla^2 f$, is the double differential form defined by $\nabla^2 f (\xi,\zeta)= \xi\zeta f - \nabla_\xi \zeta f$, where $\xi$ and $\zeta$ are vector fields and $\nabla_\xi$ is the Levi-Civit\'a connection of $(T,g)$. To make the notation consistent with the Euclidean case, we fix an orthonormal frame $\{E_i\}_{1\le i\le N}$, and let
\begin{equation*}
\begin{split}
\nabla f &= (f_1, \ldots, f_N)=(E_1f, \ldots, E_Nf),\\
\nabla^2 f &= (f_{ij})_{1\le i, j\le N}= (\nabla^2 f (E_i, E_j))_{1\le i, j\le N}.
\end{split}
\end{equation*}
Note that if $t$ is a critical point, i.e. $\nabla f(t)=0$, then $\nabla^2 f (E_i, E_j)(t)=E_iE_jf(t)$, which is similar to the Euclidean case. As in the Euclidean space, we denote by $d$ the distance function induced by Riemannian metric $g$, which is also called the geodesic distance on $(T,g)$.

If $X(\cdot)\in C^2(T)$, where $T$ is a smooth and compact manifold without boundary, and it is a Morse function a.s., then according to Corollary 9.3.5 in Adler and Taylor (2007), the Euler characteristic of the excursion set $A_u(X,T) = \{t\in T: X(t)\geq u\}$ is given by
\begin{equation}\label{Eq:def of Euler charac T}
\chi(A_u(X,T))= (-1)^N \sum^N_{i=0}(-1)^i \mu_i(T)
\end{equation}
with
\begin{equation*}
\begin{split}
\mu_i(T) & := \# \big\{ t\in T: X(t)\geq u, \nabla X(t)=0, \text{index} (\nabla^2 X(t))=i \big\}.
\end{split}
\end{equation*}
If $T$ is a convex and compact sets with smooth boundary, then we have
\begin{equation*}
\chi(A_u(X,T))= (-1)^N \sum^N_{i=0}(-1)^i \mu_i(\overset{\circ}{T}) + (-1)^{N-1} \sum^{N-1}_{i=0}(-1)^i \mu_i(\partial T)
\end{equation*}
with
\begin{equation*}
\begin{split}
\mu_i(\overset{\circ}{T}) & := \# \big\{ t\in \overset{\circ}{T}: X(t)\geq u, \nabla X(t)=0, \text{index} (\nabla^2 X(t))=i \big\},\\
\mu_i(\partial T) & := \# \big\{ t\in \partial T: X(t)\geq u, \nabla X_{|\partial T}(t)=0, \text{index} (\nabla^2 X_{|\partial T}(t))=i \big\}.
\end{split}
\end{equation*}

By similar arguments for Gaussian fields on rectangles in the previous section, together with the projection technique in Aza\"is and Delmas (2002) or the arguments by charts in Theorem 12.1.1 in Adler and Taylor (2007), we can obtain the following extension, whose proof is omitted here.
\begin{theorem} \label{Thm:EECA-T}
Let $X = \{X(t): t\in T\}$ be a Gaussian random field satisfying $({\bf H}1)$ and $({\bf H}2)$, where $T$ is a smooth and compact manifold without boundary or a convex and compact set with smooth boundary. Then there exists some $\alpha>0$ such that the expected Euler characteristic approximation \eqref{Eq:MEC approx error} holds.
\end{theorem}

\section{The Expected Euler Characteristic}
We now turn to computing the expected Euler characteristic $\E\{\chi(A_u(X,T))\}$. To do this, we need some preliminary results on calculations of certain Gaussian matrices.

\subsection{Preliminary Computations on Gaussian Matrices}
The following lemma can be obtained by elementary calculations. See also Lemma 11.6.1 in Adler and Taylor (2007) for reference.
\begin{lemma}\label{Lem:Wick formula} {\rm \textbf {(Wick formula).}} Let $(Z_1, Z_2, ..., Z_N)$ be a centered Gaussian random vector. Then for any integer $k$,
\begin{equation*}
\begin{split}
\E\{Z_1Z_2\cdots Z_{2k+1}\}&=0,\\
\E\{Z_1Z_2\cdots Z_{2k}\}&=\sum \E\{Z_{i_1}Z_{i_2}\}\cdots\E\{Z_{i_{2k-1}}Z_{i_{2k}}\},
\end{split}
\end{equation*}
where the sum is taken over the $(2k)!/(k!2^k)$ different ways of grouping $Z_1$, ..., $Z_{2k}$ into $k$ pairs.
\end{lemma}

Let $\Delta_N = (\Delta_{ij})_{1\le i, j\le N}$ and $\Xi_N = (\Xi_{ij})_{1\le i, j\le N}$ be two $N\times N$ symmetric centered Gaussian matrices satisfying the following properties:
\begin{equation}\label{Eq:symmetric property}
\begin{split}
\E\{\Delta_{ij}\Delta_{kl}\} &= \mathcal{E}(i,j,k,l) - \delta_{ij}\delta_{kl},\\
\E\{\Xi_{ij}\Xi_{kl}\} &= \mathcal{F}(i,j,k,l),
\end{split}
\end{equation}
where $\mathcal{E}$ and $\mathcal{F}$ are both symmetric function of $i$, $j$, $k$, $l$, and $\delta_{ij}$ is the Kronecker delta function.

The following result is an extension of Lemma 11.6.2 in Adler and Taylor (2007). It will be used for computing the expected Euler characteristic of stationary or isotropic Gaussian fields.

\begin{lemma}\label{Lem:det of Delta}  Let $B_N=(B_{ij})_{1\leq i,j\leq N}$ be an $N\times N$ real symmetric matrix. Then, under (\ref{Eq:symmetric property}),
\begin{equation}\label{Eq:DeltaXi}
\begin{split}
\E\{{\rm det}(\Delta_N+B_N)\}&=\sum_{k=0}^{\lfloor N/2 \rfloor} \frac{(-1)^k(2k)!}{k!2^k} S_{N-2k}(B_N),\\
\E\{{\rm det}(\Xi_N+B_N)\}&={\rm det}(B_N),
\end{split}
\end{equation}
where $S_j(B_l)$ denotes the sum of the $\binom {l}{j}$ principle minors of order $j$ in $B_l$, and $S_0(B_l)=1$ by convention.
\end{lemma}

\begin{proof}\ We first consider the case when $N$ is even, say $N=2l$. Then
\begin{equation*}
\E\{{\rm det}(\Delta_{2l}+B_{2l})\} = \sum_{P}\eta(p)\E\{(\Delta_{1i_1}+B_{1i_1})\cdots (\Delta_{2l,i_{2l}}+B_{2l,i_{2l}})\},
\end{equation*}
where $p=(i_1,i_2\cdots,i_{2l})$ is a permutation of $(1,2,\cdots, 2l)$, $P$ is the set of the $(2l)!$ such permutations, and $\eta(p)$ equals $+1$ or $-1$ depending on the order of the permutation $p$. It follows from Lemma \ref{Lem:Wick formula} that for $k\leq l$, $\E\{\Delta_{1i_1}\cdots\Delta_{2k-1,i_{2k-1}}\}=0$ and
\begin{equation*}
\begin{split}
\E\{\Delta_{1i_1}\cdots\Delta_{2k,i_{2k}}\} &= \sum_{Q_{2k}} \{\mathcal{E}(1,i_1,2,i_2)-\delta_{1i_1}\delta_{2i_2}\}\times \cdots\\
&\quad \times \{\mathcal{E}(2k-1,i_{2k-1},2k,i_{2k})-\delta_{2k-1, i_{2k-1}}\delta_{2k,i_{2k}}\},
\end{split}
\end{equation*}
where $Q_{2k}$ is the set of the $(2k)!/(k!2^k)$ ways of grouping $(i_1,i_2,\cdots,i_{2k})$ into pairs without regard to order, keeping them paired with the first index. Hence
\begin{equation*}
\begin{split}
&\quad \sum_{P}\eta(p) \E\{\Delta_{1i_1}\cdots\Delta_{2k,i_{2k}}\}B_{2k+1,i_{2k+1}}\cdots B_{2l,i_{2l}} \\
&=\sum_{P}\eta(p) \left(\sum_{Q_{2k}} \{\mathcal{E}(1,i_1,2,i_2)-\delta_{1i_1}\delta_{2i_2}\} \cdots\{\mathcal{E}(2k-1,i_{2k-1},2k,i_{2k})-\delta_{2k-1, i_{2k-1}}\delta_{2k,i_{2k}}\}\right)\\
&\quad \times B_{2k+1,i_{2k+1}}\cdots B_{2l,i_{2l}}\\
&= \sum_{P}\eta(p)\left(\sum_{Q_{2k}} (-1)^k (\delta_{1i_1}\delta_{2i_2})\cdots (\delta_{2k-1, i_{2k-1}}\delta_{2k,i_{2k}})\right)B_{2k+1,i_{2k+1}}\cdots B_{2l,i_{2l}}\\
&=\frac{(-1)^k(2k)!}{k!2^k}{\rm det}((B_{ij})_{2k+1\leq i,j\leq 2l}),
\end{split}
\end{equation*}
where the second equality is due to the fact that all products involving at least one $\mathcal{E}$ term will cancel out because of their symmetry property, and the last equality comes from changing the order of summation and then noting that the delta functions are nonzero only in those permutations in $P$ with $(i_1, i_2, \ldots, i_{2k-1}, i_{2k})=(1,2,\ldots, 2k-1, 2k)$. Thus
\begin{equation*}
\begin{split}
\E\{{\rm det}(\Delta_{2l}+B_{2l})\}=\sum_{k=0}^l \frac{(-1)^k(2k)!}{k!2^k} S_{2l-2k}(B_{2l}).
\end{split}
\end{equation*}
Similarly, we obtain that when $N=2l+1$,
\begin{equation*}
\begin{split}
\E\{{\rm det}(\Delta_{2l+1}+B_{2l+1})\}=\sum_{k=0}^l \frac{(-1)^k(2k)!}{k!2^k} S_{2l+1-2k}(B_{2l+1}).
\end{split}
\end{equation*}
The proof for the first line in \eqref{Eq:DeltaXi} is completed. The second line in \eqref{Eq:DeltaXi} follows similarly.
\end{proof}

Let $B_N(i_1,\ldots, i_n; i_1,\ldots, i_n) = (B_{i_{j}i_{k}})_{1\le j, k\le n}$ be the $n\times n$ principle submatrix of $B_N$ extracted from the $i_1,\ldots, i_n$ rows and $i_1,\ldots, i_n$ columns in $B_N$, where $1\le i_1 < \cdots < i_n \le N$.
\begin{proposition} \label{Prop:Laplace expansion} Let $\Delta_N$ and $\Xi_N$ be two $N\times N$ symmetric centered Gaussian matrices satisfying \eqref{Eq:symmetric property}, and let $B_N$ be an $N\times N$ real symmetric matrix. Then for $x\in \R$,
\begin{equation}\label{Eq:DeltaXi2}
\begin{split}
\E\{{\rm det}(\Delta_N + B_N -xI_N)\}&=\sum_{n=0}^N \frac{(-1)^{N-n}}{(N-n)!}\left(\sum_{k=0}^{\lfloor n/2 \rfloor} \frac{(-1)^k(N-n+2k)!}{k!2^k} S_{n-2k}(B_N)\right)x^{N-n},\\
\E\{{\rm det}(\Xi_N + B_N -xI_N)\}&=\sum_{n=0}^N (-1)^{N-n}S_n(B_N)x^{N-n},
\end{split}
\end{equation}
where $S_j(\cdot)$ is defined in Lemma \ref{Lem:det of Delta}.
\end{proposition}
\begin{proof}\ Applying the Laplace expansion of the determinant yields
\begin{equation}\label{Eq:LapExpan}
\E\{{\rm det}(\Delta_N + B_N -xI_N)\}=\sum_{n=0}^N (-1)^{N-n} \E\{S_n(\Delta_N + B_N)\}x^{N-n}.
\end{equation}
By Lemma \ref{Lem:det of Delta},
\begin{equation*}
\begin{split}
\E\{S_n(\Delta_N + B_N)\} &=\sum_{1\le i_1 < \cdots < i_n \le N}\sum_{k=0}^{\lfloor n/2 \rfloor} \frac{(-1)^k(2k)!}{k!2^k} S_{n-2k}(B_N(i_1,\ldots, i_n; i_1,\ldots, i_n))\\
&= \sum_{k=0}^{\lfloor n/2 \rfloor}\frac{(-1)^k(2k)!}{k!2^k} \frac{\binom{N}{n}\binom{n}{n-2k}}{\binom{N}{n-2k}}S_{n-2k}(B_N)\\
&= \frac{1}{(N-n)!}\sum_{k=0}^{\lfloor n/2 \rfloor}\frac{(-1)^k(N-n+2k)!}{k!2^k} S_{n-2k}(B_N),
\end{split}
\end{equation*}
where the second equality is due to the observation that the sum on all principle submatrices of order $n$ in the first line makes every principal minor of order $n-2k$ appear $\binom{N}{n}\binom{n}{n-2k}\big/\binom{N}{n-2k}$ many times. Plugging this into \eqref{Eq:LapExpan} yields the first line in \eqref{Eq:DeltaXi2}.

By Lemma \ref{Lem:det of Delta} again,
\begin{equation*}
\E\{S_n(\Xi_N + B_N)\} =\sum_{1\le i_1 < \cdots < i_n \le N} {\rm det}(B_N(i_1,\ldots, i_n; i_1,\ldots, i_n)) =S_n(B_N).
\end{equation*}
Plugging this into \eqref{Eq:LapExpan}, with $\Delta_N$ being replaced by $\Xi_N$, yields the second line in \eqref{Eq:DeltaXi2}.
\end{proof}

\begin{remark}
Let $B_N\equiv 0$. Then it can be derived from Proposition \ref{Prop:Laplace expansion} that
\begin{equation*}
\E\{{\rm det}(\Delta_N + B_N -xI_N)\}=(-1)^N H_N(x),
\end{equation*}
coinciding with the result in Corollary 11.6.3 in Adler and Taylor (2007), where $H_N(x)$ is the Hermite polynomial of order $N$. Meanwhile,
\begin{equation*}
\E\{{\rm det}(\Xi_N + B_N -xI_N)\}=(-1)^Nx^N.
\end{equation*}
\end{remark}

\subsection{Non-centered Stationary Gaussian Fields on Rectangles}
Let $X = \{X(t): t\in T \}$ be a Gaussian random field such $X(t) = Z(t) + m(t)$, where $Z$ is a centered unit-variance stationary Gaussian random field, $m$ is the mean function of $X$, and as usual, $T$ is a compact rectangle. By classical spectral representation for stationary Gaussian fields [cf. Chapter 5 in Adler and Taylor (2007)], the field $Z$ has representation
$$
Z(t) = \int_{\R^N} e^{i \l t, \lambda \r} W(d\lambda)
$$
and covariance
$$
C(t) = \int_{\R^N} e^{i \l t, \lambda \r} \nu(d\lambda),
$$
where $W$ is a complex-valued Gaussian random measure and $\nu$ is the spectral measure satisfying $\nu(\R^N)=C(0)=\sigma^2$. We introduce the second-order spectral moments
$$
\la_{ij}=\int_{\R^N} \la_i\la_j \nu(d\lambda),
$$
and for any face $J\in \partial_k T$ with $k\ge 1$, denote $\La_J=(\la_{ij})_{i,j\in \sigma(J)}$. Then we have
$$
{\rm Cov}(\nabla Z_{|J}(t), \nabla Z_{|J}(t))=-{\rm Cov}(Z(t), \nabla^2 Z_{|J}(t))=\La_J
$$
and that
\begin{equation*}
\mathcal{E}_0(i,j,k,l):=\E\{Z_{ij}(t)Z_{kl}(t)\}=\int_{\R^N} \la_i\la_j\la_k\la_l \nu(d\lambda)
\end{equation*}
is a symmetric function of $i$, $j$, $k$, $l$.

Recall that for a $k\times k$ positive definite matrix $B$, the \emph{principal square root} of $B^{-1}$, which is usually denoted by $B^{-1/2}$, is the unique $k \times k $ positive definite matrix $Q$ such that $QBQ = I_k$. Denote by $\Psi (x)$ the tail probability of a standard Gaussian distribution, that is $\Psi (x) = (2\pi)^{-1/2}\int_x^\infty e^{-y^2/2}dy$. Notice that in \eqref{Eq:MEC} below, for every $\{t\}\in \partial_0 T$, $\nabla X(t) \in E(\{t\})$ specifies the signs of the partial derivatives $X_j(t)$ ($j= 1, \ldots, N$) and, for $J \in \partial_k T$ with $k\ge 1$, the set $\{J_1, \ldots, J_{N-k}\}$ and $E(J)$ are defined in \eqref{Def:E(J)}.
\begin{theorem} \label{Thm:MEC}
Let $X = \{X(t): t\in T \}$ be a Gaussian random field such that $X(t) = Z(t) + m(t)$, where $Z$ is a centered unit-variance stationary Gaussian random field and $m$ is the mean function of $X$. If $X$ satisfies conditions $({\bf H}1)$ and $({\bf H}2')$, then
\begin{equation}\label{Eq:MEC}
\begin{split}
&\quad \E\{\chi(A_u(X,T))\} \\
&= \sum_{\{t\}\in \partial_0 T} \P\{\nabla X(t)\in E(\{t\})\} \Psi (u-m(t)) + \sum_{k=1}^N\sum_{J\in\partial_k T}\frac{({\rm det}(\La_J))^{1/2}}{(2\pi)^{(k+1)/2}}\\
&\quad \times \int_J dt \int^\infty_u dx\, \exp\left\{-\frac{1}{2} \left[(x-m(t))^2 + (\nabla m_{|J}(t))^T\La^{-1}_J\nabla m_{|J}(t)\right]\right\} \\
&\quad \times \P\{(X_{J_1}(t), \cdots, X_{J_{N-k}}(t))\in E(J)|\nabla X_{|J}(t)=0\}\\
&\quad \times \left[\sum_{j=0}^k \frac{(-1)^j}{(k-j)!}\left(\sum_{i=0}^{\lfloor j/2 \rfloor} \frac{(-1)^i(k-j+2i)!}{i!2^i}S_{j-2i}\left(\La^{-1/2}_J\nabla^2 m_J(t)\La^{-1/2}_J\right)\right)x^{k-j}\right],
\end{split}
\end{equation}
where $S_{j-2i}(\cdot)$ is defined in Lemma \ref{Lem:det of Delta} and $\La_J^{-1/2}$ is principal square root of $\La_J^{-1}$.
\end{theorem}

\begin{proof}\ If $J=\{t\} \in \partial_0 T$, then
\begin{equation}\label{Eq:zero dimensional face}
\begin{split}
\E\{\mu_0(J)\} &= \P\{ X(t)\geq u, \ep^*_jX_j(t) \geq 0 \ {\rm for \ all}\ 1\leq j\leq N \}\\
&=\P\{\nabla X(t)\in E(\{t\})\} \Psi (u-m(t) ),
\end{split}
\end{equation}
where the last equality is due to the independence of $X(t)$ and $\nabla X(t)$ for each fixed $t$.

Let $J\in \partial_k T$ with $k\geq 1$ and let $\mathcal{D}_i$ be the collection of all $k \times k $ matrices with index $i$. Applying the Kac-Rice metatheorem, similarly to the proof of Lemma \ref{Lem:approximateMuE}, we obtain
\begin{equation}\label{Eq:mean mu}
\begin{split}
&\quad \E\bigg\{\sum^k_{i=0} (-1)^i \mu_i (J)\bigg\} \\
&= \int_J p_{\nabla X_{|J}(t)}(0)  dt \sum^k_{i=0} (-1)^i  \E\{ |\text{det} \nabla^2 X_{|J}(t)| \mathbbm{1}_{\{\nabla^2 X_{|J}(t) \in \mathcal{D}_i\}}\\
 &\quad \times \mathbbm{1}_{\{X(t)\geq u\}}  \mathbbm{1}_{\{ (X_{J_1}(t), \cdots, X_{J_{N-k}}(t))\in E(J)\}}  | \nabla X_{|J}(t)=0 \}\\
&=\int_J p_{\nabla X_{|J}(t)}(0)  dt \, \E\{ \text{det} \nabla^2 X_{|J}(t)\mathbbm{1}_{\{X(t)\geq u\}}  \mathbbm{1}_{\{ (X_{J_1}(t), \cdots, X_{J_{N-k}}(t))\in E(J)\}}  | \nabla X_{|J}(t)=0 \}\\
 &= \frac{1}{(2\pi)^{(k+1)/2}({\rm det}(\La_J))^{1/2}}\int_J dt \int^\infty_u dx \exp\left\{-\frac{1}{2} \left[(x-m(t))^2 + (\nabla m_{|J}(t))^T\La^{-1}_J\nabla m_{|J}(t)\right]\right\}\\
 &\quad \times \P\{(X_{J_1}(t), \cdots, X_{J_{N-k}}(t))\in E(J)|\nabla X_{|J}(t)=0\} \E\{\text{det} \nabla^2 X_{|J}(t) | X(t)=x\},
\end{split}
\end{equation}
where the last equality is due to the fact that $\nabla X(t)$ is independent of both $X(t)$ and $\nabla^2 X(t)$ for each fixed $t$.

Now we turn to computing $\E\{\text{det} \nabla^2 X_{|J}(t) | X(t)=x \}$. To simplify the notation, let $Q=\La^{-1/2}_J$. Then
\begin{equation}\label{Eq:covariance of zero order and second order}
\E\{Z(t)(Q \nabla^2 Z_{|J}(t) Q)_{ij}\} = -(Q\La_JQ)_{ij} = -\delta_{ij}
\end{equation}
and we can write
\begin{equation*}
\begin{split}
\E\{\text{det} (Q \nabla^2 X_{|J}(t) Q) | X(t)=x\} &=  \E\{\text{det} (Q \nabla^2 Z_{|J}(t) Q + Q\nabla^2 m_J(t)Q) | X(t)=x\}\\
&= \E\{\text{det} (\Delta(x) + Q\nabla^2 m_J(t)Q)\},
\end{split}
\end{equation*}
where $\Delta(x) = (\Delta_{ij}(x))_{i,j\in \sigma(J)}$ is a Gaussian matrix. Applying the well-known conditional formula for Gaussian variables and (\ref{Eq:covariance of zero order and second order}), we obtain
\begin{equation*}
\begin{split}
\E \{ \Delta_{ij} (x) \} = \E\{(Q \nabla^2 Z_{|J}(t) Q)_{ij} | X(t)=x\}= -x\delta_{ij}
\end{split}
\end{equation*}
and
\begin{equation*}
\begin{split}
\E &\{ [\Delta_{ij} (x)-\E\{\Delta_{ij} (x)\}] [\Delta_{kl}(x)-\E\{\Delta_{kl} (x)\}]\}\\
&= \E\{(Q \nabla^2 Z_{|J}(t) Q)_{ij}(Q \nabla^2 Z_{|J}(t) Q)_{kl}\}-\delta_{ij}\delta_{kl}= \mathcal{E}(i,j,k,l)-\delta_{ij}\delta_{kl},
\end{split}
\end{equation*}
where $\mathcal{E}$ is a symmetric function of $i, j, k, l$. Therefore,
\begin{equation*}
\begin{split}
\E\{\text{det} (Q\nabla^2 X_{|J}(t)Q) | X(t)=x\}= \E\{\text{det} (\Delta+ Q\nabla^2 m_J(t)Q-xI_k) \},
\end{split}
\end{equation*}
where $\Delta=(\Delta_{ij})_{i,j\in \sigma(J)}$ and $\Delta_{ij}$ are Gaussian variables satisfying
 \begin{equation*}
\E\{\Delta_{ij}\}=0, \qquad \E \{\Delta_{ij} \Delta_{kl}\}= \mathcal{E}(i,j,k,l)-\delta_{ij}\delta_{kl}.
\end{equation*}
It follows from Proposition \ref{Prop:Laplace expansion} that
\begin{equation*}
\begin{split}
\E&\{\text{det} (Q\nabla^2 X_{|J}(t)Q) | X(t)=x\} \\
&= \sum_{j=0}^k \frac{(-1)^{k-j}}{(k-j)!}\left(\sum_{i=0}^{\lfloor j/2 \rfloor} \frac{(-1)^i(k-j+2i)!}{i!2^i}S_{j-2i}(Q\nabla^2 m_J(t)Q)\right)x^{k-j}.
\end{split}
\end{equation*}
Therefore,
\begin{equation*}
\begin{split}
\E&\{\text{det} \nabla^2 X_{|J}(t) | X(t)=x\} = {\rm det}(\La_J) \E\{\text{det} (Q\nabla^2 X_{|J}(t)Q) | X(t)=x\}\\
&= {\rm det}(\La_J) \sum_{j=0}^k \frac{(-1)^{k-j}}{(k-j)!}\left(\sum_{i=0}^{\lfloor j/2 \rfloor} \frac{(-1)^i(k-j+2i)!}{i!2^i}S_{j-2i}(Q\nabla^2 m_J(t)Q)\right)x^{k-j}.
\end{split}
\end{equation*}
Plugging this into \eqref{Eq:mean mu}, together with \eqref{Eq:zero dimensional face} and \eqref{Eq:def of Euler charac}, yields the desired result.
\end{proof}

\begin{corollary} \label{Cor:unique maximal point}
Let the conditions in Theorem \ref{Thm:MEC} hold. Assume additionally that $t_0$, an interior point in $T$, is the unique maximum point of $m(t)$ and that $\nabla^2 m(t_0)$ is nondegenerate. Then as $u\to \infty$,
\begin{equation}\label{Eq:asymptoticLap}
\begin{split}
\E\{\chi(A_u(X,T))\} =\frac{\sqrt{{\rm det}(\La_J)}u^{N/2}}{\sqrt{{\rm det}(-\nabla^2 m(t_0))}} \Psi(u-m(t_0)) (1+o(1)).
\end{split}
\end{equation}
\end{corollary}
\begin{proof} \ By Theorem \ref{Thm:MEC},
\begin{equation*}
\begin{split}
\E\{\chi(A_u(X,T))\} &= \frac{\sqrt{{\rm det}(\La_J)}}{(2\pi)^{(N+1)/2}} \int^\infty_u x^Ndx \\
&\quad \times \int_J \exp\left\{-\frac{1}{2} \left[(x-m(t))^2 + (\nabla m(t))^T\La^{-1}\nabla m(t)\right]\right\}dt(1+o(1)).
\end{split}
\end{equation*}
Applying the Laplace method [see, e.g., Wong (2001)], we obtain that as $x\to \infty$,
\begin{equation*}
\begin{split}
&\int_J \exp\left\{-\frac{1}{2} \left[(x-m(t))^2 + (\nabla m(t))^T\La^{-1}\nabla m(t)\right]\right\}dt\\
&\quad = \frac{(2\pi)^{N/2}}{x^{N/2}\sqrt{{\rm det}(-\nabla^2 m(t_0))}}\exp\left\{-\frac{1}{2}(x-m(t_0))^2\right\}(1+o(1)).
\end{split}
\end{equation*}
Thus as $u\to \infty$,
\begin{equation*}
\begin{split}
\E\{\chi(A_u(X,T))\} &= \frac{\sqrt{{\rm det}(\La_J)}}{\sqrt{2\pi}\sqrt{{\rm det}(-\nabla^2 m(t_0))}} \int^\infty_u x^{N/2}\exp\left\{-\frac{1}{2}(x-m(t_0))^2\right\} dx (1+o(1))\\
&=\frac{\sqrt{{\rm det}(\La_J)}u^{N/2}}{\sqrt{{\rm det}(-\nabla^2 m(t_0))}} \Psi(u-m(t_0)) (1+o(1)).
\end{split}
\end{equation*}
\end{proof}

\begin{remark}
The asymptotic approximation in \eqref{Eq:asymptoticLap} is a special case of Theorem 5 in Piterbarg and Stamatovich (1998) when the index $\alpha$ therein equals 2, which implies the Gaussian field is smooth. However, in our result, a higher-order approximation is also available by applying a higher-order Laplace approximation to $\E\{\chi(A_u(X,T))\}$ [see, e.g., Wong (2001)]. Since the calculation is tedious, it is omitted here.
\end{remark}

\begin{corollary} \label{Cor:isotropic field}
Let the conditions in Theorem \ref{Thm:MEC} hold. If $Z$ is an isotropic Gaussian random field with ${\rm Var}(Z_1(t))=\gamma^2$, then
\begin{equation*}
\begin{split}
&\quad \E\{\chi(A_u(X,T))\} \\
&= \sum_{\{t\}\in \partial_0 T} \P\{\nabla X(t)\in E(\{t\})\} \Psi (u-m(t)) + \sum_{k=1}^N\sum_{J\in\partial_k T}\frac{\gamma^k}{(2\pi)^{(k+1)/2}}\\
&\quad \times \int_J dt \int^\infty_u dx\, \exp\left\{-\frac{1}{2} \left[(x-m(t))^2 + \gamma^{-2} \|\nabla m_{|J}(t)\|^2\right]\right\} \\
&\quad \times \P\{(X_{J_1}(t), \cdots, X_{J_{N-k}}(t))\in E(J)\}\\
&\quad \times \left[\sum_{j=0}^k \frac{(-1)^j}{(k-j)!}\left(\sum_{i=0}^{\lfloor j/2 \rfloor} \frac{(-1)^i(k-j+2i)!}{i!2^i}\gamma^{-2(j-2i)}S_{j-2i}\left(\nabla^2 m_J(t)\right)\right)x^{k-j}\right].
\end{split}
\end{equation*}
\end{corollary}
\begin{proof}\ The result follows immediately from Theorem \ref{Thm:MEC}, the independence of $X_i(t)$ and $X_j(t)$ when $i\neq j$ and that, for $J \in \partial_k T$ with $k \ge 1$, $\La_J=\gamma^2 I_k$ , which implies $\La_J^{-1/2}=\gamma^{-1}I_k$.
\end{proof}

\subsection{Non-centered Isotropic Gaussian Fields on Spheres}
Let $\mathbb{S}^N$ denote the $N$-dimensional unit sphere and let $X = \{X(t): t\in \mathbb{S}^N\}$ be a Gaussian random field such $X(t) = Z(t) + m(t)$, where $Z$ is a centered unit-variance isotropic Gaussian random field on $\mathbb{S}^N$ and $m$ is the mean function of $X$.

The following theorem by Schoenberg (1942) characterizes the covariances of isotropic Gaussian fields on $\mathbb{S}^N$ [see also Gneiting (2013)].
\begin{theorem}\label{Thm:Schoenberg} A continuous function $C(\cdot, \cdot):\mathbb{S}^N\times \mathbb{S}^N \rightarrow \R $ is the covariance of an isotropic Gaussian field on $\mathbb{S}^N$ if and only if it has the form
\begin{equation}\label{Eq:Schoenberg}
C(t,s)= \sum_{n=0}^\infty a_n P_n^\la(\l t, s \r), \quad  t, s \in \mathbb{S}^N,
\end{equation}
where $\la=(N-1)/2$, $a_n \geq 0$, $\sum_{n=0}^\infty a_nP_n^\la(1) <\infty$, and $P_n^\la$ is the ultraspherical polynomials defined by the expansion
\begin{equation*}
(1-2rx +r^2)^{-\la}=\sum_{n=0}^\infty r^n P_n^\la(x), \quad x\in [-1,1].
\end{equation*}
\end{theorem}

If $X$ is centered,  then it only depends on the covariance function which behaves isotropically over $\mathbb{S}^N$. Therefore, as discussed in Cheng and Xiao (2014a) and Cheng and Schwartzman (2015), we do not need to introduce any specific coordinate on sphere. However, if $X$ is non-centered, then then mean function $m$ can be very general and hence it is much more convenient to use the usual spherical coordinates for arguments, especially for obtaining exact asymptotics. To achieve this, for $t=(t_1, \cdots, t_{N+1})\in \mathbb{S}^N$, we define the corresponding spherical coordinate $\theta=(\theta_1, \ldots, \theta_N)$ as follows.
\begin{equation*}
\begin{split}
t_1 &= \cos \theta_1,\\
t_2 &= \sin \theta_1 \cos \theta_2,\\
t_3 &= \sin \theta_1 \sin \theta_2 \cos \theta_3,\\
& \vdots\\
t_{N} &= \sin \theta_1 \sin \theta_2 \cdots \sin \theta_{N-1} \cos \theta_{N},\\
t_{N+1} &= \sin \theta_1 \sin \theta_2 \cdots \sin \theta_{N-1} \sin \theta_{N},
\end{split}
\end{equation*}
where $\theta\in \Theta :=[0, \pi]^{N-1}\times[0, 2 \pi)$. We also define the Gaussian random field $\wt{X} = \{ \wt{X}(\theta): \theta \in \Theta\}$ by $\wt{X}(\theta) = X(t)$ and denote by $\wt{C}$ the covariance function of $\wt{X}$ accordingly. Similarly, let $\wt{Z}(\theta)=Z(t)$ and $\wt{m}(\theta)=m(t)$.

Lemma \ref{Lem:derivatives of covariance} below, characterizing the covariance of $(\wt{X}(\theta), \nabla \wt{X}(\theta), \nabla^2 \wt{X}(\theta))$, can be obtained easily by elementary calculations. The proof is omitted here.
\begin{lemma}\label{Lem:derivatives of covariance}
Let $X = \{X(t): t\in \mathbb{S}^N \}$ be a non-centered isotropic Gaussian random field with covariance \eqref{Eq:Schoenberg} and satisfying $({\bf H}1)$ and $({\bf H}2')$. Then
\begin{equation*}
\begin{split}
\frac{\partial \wt{C}(\theta,\varphi)}{\partial \theta_i}|_{\theta=\varphi} &= \frac{\partial^3 \wt{C}(\theta,\varphi)}{\partial \theta_i\partial \varphi_j\partial \varphi_k}|_{\theta=\varphi} =0,\\
\frac{\partial^2 \wt{C}(\theta,\varphi)}{\partial \theta_i\partial \varphi_j}|_{\theta=\varphi}&=-\frac{\partial^2 \wt{C}(\theta,\varphi)}{\partial \theta_i\partial \theta_j}|_{\theta=\varphi}=C'\delta_{ij},\\
\frac{\partial^4 \wt{C}(\theta,\varphi)}{\partial \theta_i\partial \theta_j \partial \varphi_k \partial \varphi_l}|_{\theta=\varphi}&= C''(\delta_{ij}\delta_{kl}+\delta_{ik}\delta_{jl}+\delta_{il}\delta_{jk}) + C' \delta_{ij}\delta_{kl},
\end{split}
\end{equation*}
where
\begin{equation}\label{Def:C' and C''}
C'=\sum_{n=1}^\infty a_n\left(\frac{d}{dx}P_n^\la(x)|_{x=1}\right) \quad \text{\rm and } \quad C''=\sum_{n=2}^\infty a_n\left(\frac{d^2}{dx^2}P_n^\la(x)|_{x=1}\right).
\end{equation}
\end{lemma}

Now we can formulate the expected Euler characteristic of non-centered Gaussian fields on sphere as follows.
\begin{theorem}\label{Thm:MECSphere}
Let $X = \{X(t): t\in \mathbb{S}^N \}$ be a Gaussian random field such that $X(t) = Z(t) + m(t)$, where $Z$ is a centered unit-variance isotropic Gaussian random field on $\mathbb{S}^N$ with covariance \eqref{Eq:Schoenberg} and $m$ is the mean function of $X$. If $X$ satisfies conditions $({\bf H}1)$ and $({\bf H}2')$, then
\begin{equation}\label{Eq:MECSphere}
\begin{split}
&\quad\E\{\chi(A_u(X,\mathbb{S}^N))\}\\
&= \frac{1}{(2\pi)^{(N+1)/2}} \int_{\Theta} \phi(\theta)d\theta \int_u^\infty dx \exp\left\{-\frac{1}{2} \left[(x-\wt{m}(\theta))^2 + (C')^{-1}\|\nabla \wt{m}(\theta))\|^2\right]\right\}\\
&\quad \times \left[ \sum_{j=0}^N \frac{(-1)^j}{(N-j)!}\left(\sum_{i=0}^{\lfloor j/2 \rfloor} \frac{(-1)^i(N-j+2i)!}{i!2^i}(C')^{\frac{N}{2}-j+i}(C'-1)^iS_{j-2i}\left(\nabla^2 \wt{m}(\theta)\right)\right) x^{N-j}\right],
\end{split}
\end{equation}
where $\Theta =[0, \pi]^{N-1}\times[0, 2 \pi)$, $\phi(\theta) = \prod_{i=1}^{N-1}(\sin \theta_i)^{N-i}$, and $C'$ and $S_{j-2i}(\cdot)$ are defined respectively in \eqref{Def:C' and C''}  and Lemma \ref{Lem:det of Delta}.
\end{theorem}
\begin{proof} \ Since $\mathbb{S}^N$ is a smooth and compact manifold without boundary, it follows from \eqref{Eq:def of Euler charac T}, Lemma \ref{Lem:derivatives of covariance} and the Kac-Rice metatheorem that
\begin{equation}\label{Eq:KacRiceSphere}
\begin{split}
&\quad E\{\chi(A_u(X,\mathbb{S}^N))\}\\
&= (-1)^N \int_{\Theta} \phi(\theta)d\theta \int_u^\infty dx p_{\nabla \wt{X}(\theta)}(0) p_{\wt{X}(\theta)}(x)\E\{ {\rm det} \nabla^2 \wt{X}(\theta)| \wt{X}(\theta)=x \}\\
&= \frac{(-1)^N}{(2\pi)^{(N+1)/2}(C')^{N/2}} \int_{\Theta} \phi(\theta)d\theta \int_u^\infty dx \exp\left\{-\frac{1}{2} \left[(x-\wt{m}(\theta))^2 + C'^{-1}\|\nabla \wt{m}(\theta))\|^2\right]\right\}  \\
&\quad \times \E\{ {\rm det} \nabla^2 \wt{X}(\theta)| \wt{X}(\theta)=x \}.
\end{split}
\end{equation}
We only need to compute $\E\{ {\rm det} \nabla^2 \wt{X}(\theta)| \wt{X}(\theta)=x \}$.

Case 1: $C'>1$. By Lemma \ref{Lem:derivatives of covariance}, similarly to the proof of Theorem \ref{Thm:MEC}, we get
\begin{equation*}
\begin{split}
&\quad \E\{ {\rm det} \nabla^2 \wt{X}(\theta)| \wt{X}(\theta)=x \}=\E\{ {\rm det} [\nabla^2 \wt{Z}(\theta)+ \nabla^2 \wt{m}(\theta)] | \wt{X}(\theta)=x \}\\
&=(C'^2-C')^{N/2}\E\{ {\rm det}[(C'^2-C')^{-1/2}\nabla^2 \wt{Z}(\theta)+ (C'^2-C')^{-1/2}\nabla^2 \wt{m}(\theta)] | \wt{X}(\theta)=x \}\\
&=(C'^2-C')^{N/2}\E\{ {\rm det} [\Delta + (C'^2-C')^{-1/2}\nabla^2 \wt{m}(\theta) - C'(C'^2-C')^{-1/2}x I_N]\},
\end{split}
\end{equation*}
where $\Delta=(\Delta_{ij})_{1\leq i,j\leq N}$ and $\Delta_{ij}$ are centered Gaussian variables satisfying
\begin{equation*}
\begin{split}
\E \{\Delta_{ij} \Delta_{kl}\}&= (C'^2-C')^{-1}\E\{\wt{Z}_{ij}(\theta)\wt{Z}_{kl}(\theta)|\wt{X}(\theta)=x\}\\
&=(C'^2-C')^{-1} [C''(\delta_{ij}\delta_{kl}+\delta_{ik}\delta_{jl}+\delta_{il}\delta_{jk}) + C' \delta_{ij}\delta_{kl} - C'^2 \delta_{ij}\delta_{kl} ]\\
&= \mathcal{E}(i,j,k,l)-\delta_{ij}\delta_{kl},
\end{split}
\end{equation*}
and $\mathcal{E}$ is a symmetric function of $i, j, k, l$. It then follows from Proposition \ref{Prop:Laplace expansion} that
\begin{equation}\label{Eq:DetSphere1}
\begin{split}
&\quad \E\{ {\rm det} \nabla^2 \wt{X}(\theta)| \wt{X}(\theta)=x \}\\
&= (C'^2-C')^{N/2} \sum_{j=0}^N \frac{(-1)^{N-j}}{(N-j)!}\left(\sum_{i=0}^{\lfloor j/2 \rfloor} \frac{(-1)^i(N-j+2i)!}{i!2^i}S_{j-2i}\left(\frac{\nabla^2 \wt{m}(\theta)}{\sqrt{C'^2-C'}}\right)\right)\\
&\quad \times \left(\frac{C' x}{\sqrt{C'^2-C'}}\right)^{N-j}.
\end{split}
\end{equation}

Case 2: $C'<1$. It follows from similar discussions in the previous case and a slightly revised version of Proposition \ref{Prop:Laplace expansion} that
\begin{equation}\label{Eq:DetSphere2}
\begin{split}
&\quad \E\{ {\rm det} \nabla^2 \wt{X}(\theta)| \wt{X}(\theta)=x \}\\
&= (C'-C'^2)^{N/2} \sum_{j=0}^N \frac{(-1)^{N-j}}{(N-j)!}\left(\sum_{i=0}^{\lfloor j/2 \rfloor} \frac{(N-j+2i)!}{i!2^i}S_{j-2i}\left(\frac{\nabla^2 \wt{m}(\theta)}{\sqrt{C'-C'^2}}\right)\right)\\
&\quad \times \left(\frac{C' x}{\sqrt{C'-C'^2}}\right)^{N-j}.
\end{split}
\end{equation}

Case 3: $C'=1$. By Lemma \ref{Lem:derivatives of covariance} again,
\begin{equation*}
\begin{split}
\E\{ {\rm det} \nabla^2 \wt{X}(\theta)| \wt{X}(\theta)=x \}&=\E\{ {\rm det} (\nabla^2 \wt{Z}(\theta)+ \nabla^2 \wt{m}(\theta)) | \wt{X}(\theta)=x \}\\
&=\E\{ {\rm det} (\Xi + \nabla^2 \wt{m}(\theta) - x I_N)\},
\end{split}
\end{equation*}
where $\Xi=(\Xi_{ij})_{1\leq i,j\leq N}$ and $\Xi_{ij}$ are centered Gaussian variables satisfying
\begin{equation*}
\begin{split}
\E \{\Xi_{ij} \Xi_{kl}\}= \E\{\wt{Z}_{ij}(\theta)\wt{Z}_{kl}(\theta)|\wt{X}(\theta)=x\}=C''(\delta_{ij}\delta_{kl}+\delta_{ik}\delta_{jl}+\delta_{il}\delta_{jk})
= \mathcal{F}(i,j,k,l),
\end{split}
\end{equation*}
and $\mathcal{F}$ is a symmetric function of $i, j, k, l$. It then follows from Proposition \ref{Prop:Laplace expansion} that
\begin{equation}\label{Eq:DetSphere3}
\begin{split}
\E\{ {\rm det} \nabla^2 \wt{X}(\theta)| \wt{X}(\theta)=x \}= \sum_{j=0}^N (-1)^{N-j}S_j\left(\nabla^2 \wt{m}(\theta)\right)x^{N-j}.
\end{split}
\end{equation}

Plugging respectively \eqref{Eq:DetSphere1}, \eqref{Eq:DetSphere2} and \eqref{Eq:DetSphere3} into \eqref{Eq:KacRiceSphere}, we see that the expected Euler characteristic for all three cases above can be formulated by the same expression \eqref{Eq:MECSphere}.
\end{proof}

\begin{remark}
Let $\wt{m}(\theta)\equiv 0$. Let $\omega_j =   \frac{2\pi^{(j+1)/2}}{\Gamma((j+1)/{2})}$ be the spherical area of the $j$-dimensional unit sphere. Notice that Hermite polynomials have the following properties:
\begin{equation*}
\begin{split}
\int_u^\infty H_n(x)e^{-x^2/2}dx &= H_{n-1}(u)e^{-u^2/2}, \\
x^n &= n!\sum_{k=0}^{\lfloor n/2 \rfloor} \frac{1}{k!2^k(n-2k)!} H_{n-2k}(x),
\end{split}
\end{equation*}
where $n\ge 0$ and $H_{-1}(x)=\sqrt{2\pi}\Psi(x)e^{x^2/2}$. Applying Theorem \ref{Thm:MECSphere}, together with the properties above and certain combinatorial tricks, we obtain
\begin{equation*}
\begin{split}
\E\{\chi(A_u(X,\mathbb{S}^N))\}&= \frac{\omega_N}{(2\pi)^{(N+1)/2}} \sum_{n=0}^{\lfloor N/2 \rfloor} (C')^{(N-2n)/2}\binom{N}{2n}(2n-1)!!H_{N-2n-1}(u)e^{-u^2/2}\\
&=\sum_{j=0}^N  (C')^{j/2} \mathcal{L}_j (\mathbb{S}^N) \rho_j(u),
\end{split}
\end{equation*}
where $\rho_0(u)= \Psi(u)$, $\rho_j(u) = (2\pi)^{-(j+1)/2} H_{j-1} (u) e^{-u^2/2}$ for $j\geq 1$ and
\begin{equation*}
\begin{split}
\mathcal{L}_j (\mathbb{S}^N) = \left\{
  \begin{array}{l l}
     2 \binom{N}{j}\frac{\omega_N}{\omega_{N-j}} & \quad \text{if $N-j$ is even,}\\
     0 & \quad \text{otherwise}
   \end{array} \right.
\end{split}
\end{equation*}
(for $j=0, 1, \ldots, N$) are the Lipschitz-Killing curvatures of $\mathbb{S}^N$ [cf. Eq. (6.3.8) in Adler and Taylor (2007)]. This coincides with the formula of the expected Euler characteristic for centered isotropic Gaussian fields on sphere obtained in Cheng and Xiao (2014a) via a geometric approach. However, the result in Cheng and Xiao (2014a) is still more general for studying centered isotropic Gaussian fields on sphere since it is also applicable when the parameter sets are subsets of $\mathbb{S}^N$.
\end{remark}

\bibliographystyle{plain}
\begin{small}

\end{small}

\bigskip

\begin{quote}
\begin{small}

\textsc{Dan Cheng}: Department of Statistics, North Carolina State
University, 2311 Stinson Drive, Campus Box 8203, Raleigh, NC 27695, U.S.A.\\
E-mail: \texttt{dcheng2@ncsu.edu}
\end{small}
\end{quote}

\end{document}